\documentclass[11pt]{article}
\usepackage{amsmath,amssymb,amsthm}
\usepackage[margin=2.5cm]{geometry}
\usepackage{hyperref}

\newtheorem{theorem}{Theorem}[section]
\newtheorem{lemma}[theorem]{Lemma}
\newtheorem{proposition}[theorem]{Proposition}
\newtheorem{corollary}[theorem]{Corollary}
\theoremstyle{definition}
\newtheorem{definition}[theorem]{Definition}
\newtheorem{remark}[theorem]{Remark}

\newcommand{\R}{\mathbb{R}}
\newcommand{\Sym}{\mathbb{S}}
\newcommand{\SymA}{\operatorname{Sym}}
\newcommand{\Th}{\Theta}
\newcommand{\tr}{\operatorname{tr}}
\newcommand{\cyc}{\overset{\mathrm{cyc}}{\sim}}
\newcommand{\ab}[1]{\check{#1}}

\title{Homogeneous symmetric trace-positive sextics\\
with vanishing abelianization\\
are sums of Hermitian squares and commutators}
\author{Daniel Keren\thanks{Department of Computer Science, University of
Haifa, Haifa 3498838, Israel. Email: \texttt{dkeren@ds.haifa.ac.il}}}
\date{}

\begin{document}
\maketitle

\begin{abstract}
Let $\R\langle X,Y\rangle$ be the free algebra on two symmetric letters. A
polynomial is \emph{trace-positive} if it has non-negative trace on every pair
of real symmetric matrices, and it lies in the cone $\Th$ if it is cyclically
equivalent to a sum of Hermitian squares. Burgdorf and Klep proved that the
two notions coincide in degree at most four; from degree six onwards they
differ. We provide a complete and explicit description of
both cones on the six-dimensional space $Z$ of \emph{symmetric} homogeneous
sextics in two symmetric letters whose abelianization vanishes identically,
and show that on
$Z$ they coincide: for $f\in Z$, trace-positivity is equivalent to membership
in $\Th_3$ (the sub-cone of elements cyclically equivalent to a sum of
Hermitian squares of polynomials of degree at most three), to non-negativity of
the trace on pairs of $2\times2$ symmetric
matrices, and to being cyclically equivalent to a sum of four Hermitian
squares of cubics with vanishing abelianization. 
Membership in $\Th$ reduces to a $4\times4$ positive semidefinite Gram matrix
whose two diagonal blocks are fixed by the content-$(4,2)$ and $(2,4)$ parts
of $f$; each of those parts ranges over a two-dimensional cone cut out by two
linear facets, and the content-$(3,3)$ part is then confined to a rectangle
whose two sides are the geometric means $\sqrt{p_+q_+}$ and
$\sqrt{p_-q_-}$ of the eigenvalues $p_\pm$, $q_\pm$ of those blocks.
Each facet is certified by one of two 
families of pairs of $2\times2$ symmetric matrices. The hypothesis on the abelianization is the natural
one: every bivariate degree-six example we are aware of -- a trace-positive
polynomial outside $\Th$ -- comes from the commutative gap between non-negative
and sum-of-squares polynomials, by lifting a non-negative, non-SOS polynomial in
two variables to a cyclically
sorted polynomial; and that lifting cannot operate when the abelianization
vanishes. Symmetry of $f$ is essential: dropping it adds one antisymmetric direction,
on which the trace vanishes identically, yet every non-zero element of which
lies outside $\Th$ (Remark~\ref{rem:needsym}).
\end{abstract}

\medskip
\noindent\textbf{2020 Mathematics Subject Classification.}
Primary 13J30; Secondary 11E25, 15A45, 47A57, 90C22.

\medskip
\noindent\textbf{Key words and phrases.}
Trace-positive polynomials, sum of Hermitian squares, cyclic equivalence,
free algebra, tracial moment problem, semidefinite programming.

\section{Introduction}

\subsection{The two cones}

Write $\R\langle X,Y\rangle$ for the free unital $\R$-algebra on two letters,
equipped with the involution fixing $X$ and $Y$ and reversing products, and
write $\Sym_n(\R)$ for the space of real symmetric $n\times n$ matrices. Two
polynomials are \emph{cyclically equivalent}, written $f\cyc g$, if $f-g$ is a
sum of commutators; equivalently, if for each cyclic class of words, the total coefficient of $f$ on
that class equals that of $g$. Let
\[
  \Th_d \;=\; \Big\{ f \;:\; f \cyc \textstyle\sum_i g_i^* g_i, \ \deg g_i \le d \Big\},
  \qquad \Th=\bigcup_d \Th_d ,
\]
and call $f$ \emph{trace-positive} if $\tr f(A,B)\ge 0$ for all
$A,B\in\Sym_n(\R)$ and all $n$. Since $\tr(g^*g)\ge0$ and traces annihilate
commutators,
\begin{equation}\label{eq:inclusion}
  \Th \;\subseteq\; \{\text{trace-positive}\},
\end{equation}
and the question of when the inclusion is strict is the tracial analogue of
Hilbert's problem on non-negative polynomials versus sums of squares. 
For background on $\Th$, trace-positivity and the associated semidefinite
relaxations, see the monograph of Burgdorf, Klep and Povh~\cite{BKP}.
Deciding membership in $\Th$ is dual to a truncated tracial moment
problem~\cite{BK-JOT}; Quarez~\cite{Quarez} gives trace-positive polynomials
of degree four in three variables lying outside $\Th$; and a global Positivstellensatz for trace-positive
polynomials, with its bearing on Connes' embedding problem, is given by Klep,
Scheiderer and Vol\v{c}i\v{c}~\cite{KSV}.

Two facts frame what follows. First, there is no gap in low degree:
Burgdorf and Klep~\cite{BK-CR} proved that for $f\in\R\langle X,Y\rangle$ of
degree at most four the following are equivalent -- $f$ is trace-positive;
$f$ is trace-positive on $2\times2$ symmetric matrices; $f$ is cyclically
equivalent to a sum of four Hermitian squares; $f\in\Th$. Second, the gap
appears at degree six, the standard witness being the noncommutative Motzkin
polynomial
\begin{equation}\label{eq:ncmotzkin}
  M_{nc} \;=\; 1 + X^2Y^4 + Y^2X^4 - 3X^2Y^2 ,
\end{equation}
which is the cyclic-equivalence normal form of the polynomial
$YX^4Y+XY^4X-3XY^2X+1$ of Klep and Schweighofer~\cite[Ex.\ 4.4]{KS}.

\subsection{Why the abelianization}

The reason $M_{nc}\notin\Th$ is worth stating precisely, because it motivates the hypothesis we impose. Following~\cite[Def.\ 4.1]{KS}, call $f$ \emph{cyclically
sorted} if every word occurring in $f$ is cyclically equivalent to a word of
the form $X^aY^b$. Let $\ab f\in\R[x,y]$ denote the \emph{abelianization} (or
\emph{commutative collapse}) of $f$, obtained by letting the letters commute.

\begin{proposition}[{Burgdorf~\cite{Bur-thesis}, Prop.\ 3.3--3.4}]
\label{prop:sorted}
If $f$ is cyclically sorted then $f$ is trace-positive if and only if $\ab f$
is non-negative on $\R^2$, and $f\in\Th$ if and only if $\ab f$ is a sum of
squares in $\R[x,y]$.
\end{proposition}

The first equivalence is also part of~\cite[Cor.\ 5.4]{KSV}. We stress that
the second is not: the fourth clause of that corollary places a cyclically
sorted trace-positive $f$ in the cone $\Th^{\mathrm{rat}}$ of sums of
Hermitian squares and commutators of noncommutative \emph{rational} functions
with denominators $1+x_i^2$, and $\Th^{\mathrm{rat}}$ is strictly
larger than $\Th$
-- it contains $M_{nc}$ \cite[Ex.\ 3.2]{KSV}, which is the reason for
introducing it.

One half of the second equivalence requires no hypothesis: abelianization is a ring homomorphism which annihilates commutators and sends
$g^*g$ to $(\ab g)^2$, so $f\in\Th$ forces $\ab f$ to be a sum of squares.

Consequently~\cite[Rem.\ 3.6]{Bur-thesis}: for \emph{any} positive polynomial
$p\in\R[x,y]$ that is not a sum of squares, replacing each monomial $x^ay^b$
by $X^aY^b$ gives a cyclically sorted $f\in\R\langle X,Y\rangle$ with
$\ab f=p$, and Proposition~\ref{prop:sorted} then makes $f$ trace-positive but
outside $\Th$. 
Applying this to the Motzkin polynomial yields ~\eqref{eq:ncmotzkin}, so
that $M_{nc}\notin\Th$ is a restatement of the classical fact that the Motzkin
polynomial is non-negative but not a sum of squares.

This suggests isolating the part of the problem where the construction cannot
operate. If $\ab f\equiv 0$ then there is no commutative polynomial to substitute back,
and $f$ contributes nothing on commuting pairs; so if such an $f$ were
trace-positive but outside $\Th$, the obstruction would be purely
noncommutative. We show that no such obstruction occurs for symmetric homogeneous bivariate
sextics; without the symmetry hypothesis one does occur, see
Remark~\ref{rem:needsym}.

\begin{remark}\label{rem:origin}
Every bivariate degree-six obstruction we have found in the literature arises
by the lifting construction just described. We do not know whether one of a
different origin exists, and we do not claim the list is exhaustive; the
observation is offered as motivation for the hypothesis
$\ab f\equiv0$, not as a theorem.
\end{remark}

\subsection{Result}

Let
\begin{equation}\label{eq:Z}
  Z \;=\; \{ f\in\SymA\R\langle X,Y\rangle : f \text{ homogeneous of degree }6,\
  \ab f \equiv 0 \} \big/ \!\cyc ,
\end{equation}
a real vector space of dimension six (Lemma~\ref{lem:dimZ}). Cyclic
equivalence does not preserve symmetry, but nothing is lost by quotienting the
symmetric part: if $f\cyc f^*$ then $\tfrac12(f+f^*)$ is symmetric and
cyclically equivalent to $f$, so $Z$ is also the $+1$ eigenspace of the
involution on the degree-six homogeneous polynomials with vanishing
abelianization modulo $\cyc$, which is the description used in
\S\ref{sec:words}. The symmetry
requirement $f^*=f$ is essential; Remark~\ref{rem:needsym} exhibits a
trace-zero element of the larger quotient that lies outside $\Th$.

\begin{theorem}\label{thm:main}
Let $f\in Z$. The following are equivalent.
\begin{enumerate}
  \item[(i)] $f$ is trace-positive;
  \item[(ii)] $\tr f(A,B)\ge0$ for all $A,B\in\Sym_2(\R)$;
  \item[(iii)] $f\in\Th_3$;
  \item[(iv)] $f\cyc\sum_{i=1}^{4}g_i^*g_i$ with each $g_i$ homogeneous of
        degree three and $\ab{g_i}\equiv0$.
\end{enumerate}
Moreover the common cone is described in closed form by
\eqref{eq:42cone} for each of the two extreme content sectors together with
the rectangle \eqref{eq:rect} for the coupling between them; in coordinates
this is the six inequalities of Corollary~\ref{cor:explicit}.
\end{theorem}

By Corollary~\ref{cor:gram} one may equally write $f\in\Th$ in clause (iii):
on $Z$ the two cones $\Th$ and $\Th_3$ agree.

The four clauses are in deliberate parallel with~\cite[Thm.\ 3.1]{BK-CR}: the
degree-four theorem has exactly this shape, including the reduction to
$2\times2$ matrices and the bound of four squares. Burgdorf and Klep present
that theorem as the tracial analogue of Hilbert's classical result on
non-negative ternary quartics, i.e.\ bivariate polynomials of degree four; in
the same spirit Theorem~\ref{thm:main} is a tracial statement about bivariate
sextics, the first degree at which, in two variables, non-negativity and sums
of squares part company.
Theorem~\ref{thm:main} says that the same four-clause description persists at
degree six, the first degree at which the general statement fails, once the
abelianization is assumed to vanish.

On the 
$\Th$ side of the equivalence the problem becomes a $4\times4$ positive semidefinite Gram matrix
$H=\left(\begin{smallmatrix}P&R\\ R^\top&Q\end{smallmatrix}\right)$; the two
diagonal blocks are pinned down by the extreme content sectors up to a
one-parameter degree of freedom each, and the set of off-diagonal blocks compatible with
them turns out to be a rectangle whose thresholds are the geometric means
$\sqrt{p_\pm q_\pm}$ of the eigenvalues of $P$ and $Q$. That the thresholds
are geometric rather than arithmetic means is what forces the certificates on
the trace side to scale $A$ and $B$ independently; see Remark~\ref{rem:geometric}.

In coordinates, the cone is cut out by six inequalities. Write $(s,u)$ for the
content-$(4,2)$ part of $f$, $(S,U)$ for the content-$(2,4)$ part and
$(\alpha,\beta)$ for the content-$(3,3)$ part, in the parametrizations fixed
in Sections~\ref{sec:42} and~\ref{sec:33}. Then $f$ is trace-positive -- and
therefore lies in $\Th_3$ -- if and only if
\[
  4s+u\ \ge\ 0,\qquad u\ \le\ 0,\qquad 4S+U\ \ge\ 0,\qquad U\ \le\ 0,
\]
\[
  16(\alpha+\beta)^2\ \le\ (4s+u)(4S+U),
  \qquad
  16(\alpha-\beta)^2\ \le\ uU .
\]
This is Corollary~\ref{cor:explicit}.

\subsection{What this does and does not settle}\label{sec:settle}

Theorem~\ref{thm:main} does not determine the least degree $d_{\min}$ at which a
symmetric homogeneous trace-positive element can fail to lie in $\Th$ while
having sum-of-squares abelianization; \cite{BK-CR} shows that degree is greater than
four, and $S_{6,3}(X^2,Y^2)\notin\Th$ \cite[Ex.\ 3.5]{KS-BMV} shows it is at
most twelve. What the theorem closes is the route on which the abelianization
vanishes outright, which by Remark~\ref{rem:origin} is the source of every
bivariate degree-six obstruction we know of.

\subsection{Outline of the proof}\label{sec:outline}

The proof matches two descriptions of the same cone, one by Hermitian squares
and one by traces, one content sector at a time.

\emph{Step 1: the $\Th$ side is a $4\times4$ Gram matrix}
(Section~\ref{sec:kred}). Degree and content bookkeeping reduces every
certificate $f\cyc\sum_i g_i^*g_i$ to one in which each $g_i$ lies in the
four-dimensional space $K$ of cubics with vanishing abelianization
(Lemma~\ref{lem:kred}). Hence $\Th\cap Z$ is the image of the positive
semidefinite $4\times4$ matrices $H=\left(\begin{smallmatrix}P&R\\R^\top&Q
\end{smallmatrix}\right)$, where $P$ produces the content-$(4,2)$ part of $f$,
$Q$ the $(2,4)$ part and $R$ the $(3,3)$ part; and $H\succeq0$ if and only if
$P,Q\succeq0$ and $R=P^{1/2}CQ^{1/2}$ with $\|C\|\le1$
(Corollary~\ref{cor:gram}).

\emph{Step 2: the trace side splits the same way}
(Section~\ref{sec:sectors}). Rescaling $X\mapsto\lambda X$, $Y\mapsto\mu Y$
shows that $f$ is trace-positive if and only if $T_{42}\ge0$, $T_{24}\ge0$ and
$|T_{33}|\le2\sqrt{T_{42}T_{24}}$ (Lemma~\ref{lem:decouple}) --- the same shape
as the block condition of Step~1. It remains to match the pieces.

\emph{Step 3: the extreme sectors} (Section~\ref{sec:42}). The content-$(4,2)$
part of $Z$ is two-dimensional. On the $\Th$ side, the parts arising from some
$P\succeq0$ form a cone with two linear facets, \eqref{eq:42cone}; on the trace
side, one explicit pair of $2\times2$ matrices per facet shows that
trace-positivity forces the same cone (Proposition~\ref{prop:42}). The sector
$(2,4)$ follows by exchanging $X$ and $Y$.

\emph{Step 4: the middle sector} (Section~\ref{sec:33}). The content-$(3,3)$
part is two-dimensional, with coordinates $(\alpha,\beta)$. For the $P$ and $Q$
of Step~3, the attainable $(\alpha,\beta)$ form a rectangle whose half-sides
are the geometric means $\sqrt{p_+q_+}$, $\sqrt{p_-q_-}$ of the eigenvalues of
$P$ and $Q$ (Lemma~\ref{lem:rectangle}). Conversely, two 
families of $2\times2$ pairs give negative trace at every
$(\alpha,\beta)$ outside the rectangle (Proposition~\ref{prop:33}).

\emph{Step 5: assembly} (Section~\ref{sec:proof}). Steps~3 and~4 match the two
sides exactly, which proves Theorem~\ref{thm:main}. Every witness is a pair of
$2\times2$ matrices, which gives clause~(ii), and $H$ is $4\times4$, which gives
the bound of four squares in clause~(iv).

Section~\ref{sec:prelim} fixes notation. Section~\ref{sec:union} is not used
in this argument: it verifies directly on
the $\Th$ side that no other choice of the Gram matrix enlarges the rectangle
(Remark~\ref{rem:notneeded}). Section~\ref{sec:open} lists open questions.

\section{Preliminaries}\label{sec:prelim}

\subsection{Words, cyclic classes, involution}\label{sec:words}

Let $\langle X,Y\rangle$ be the set of words. For a word $w$ let $w^*$ be its
reversal. Since $\tr(w(A,B))=\tr(w^*(A,B))$ for symmetric $A,B$, the trace of
$f$ depends only on the coefficients of $f$ summed over the classes of the
equivalence generated by cyclic rotation \emph{and} reversal; these are the
\emph{bracelet classes}, and we write $\langle w\rangle$ for the class of $w$.
Throughout, a symbol such as $c_1$ or $d_1$ denotes a bracelet class, and the
coefficient attached to it is the total coefficient mass of $f$ on that class;
the trace of the class at a pair $(A,B)$ is the trace of any one of its words.

Two quotients are in play and must not be confused. Cyclic equivalence
identifies $f$ and $g$ exactly when their coefficient sums agree on every
\emph{cyclic} class, so the homogeneous polynomials of degree six modulo
$\cyc$ form a space whose dimension is the number of cyclic classes, namely
fourteen. The involution descends to that space, fixing twelve of the classes
and interchanging the classes of $X^2YXY^2$ and $X^2Y^2XY$. Its $+1$
eigenspace -- equivalently the image of the symmetric polynomials
$\SymA\R\langle X,Y\rangle=\{f:f^*=f\}$ -- therefore has dimension thirteen,
and on it the coefficient sums of those two classes coincide, so that
\emph{bracelet classes are exactly the coordinates on the symmetric part}.
The $-1$ eigenspace is the line spanned by the difference of those two
classes, and the trace annihilates it. Everything below takes place on the
symmetric part; see Remark~\ref{rem:needsym} for why the restriction is
necessary and not merely convenient.

The \emph{content} of a word is the pair $(\deg_X w,\deg_Y w)$. Content is
additive under concatenation and invariant under rotation and reversal; in
particular it is defined on bracelet classes, and
$\mathrm{content}(u^*v)=\mathrm{content}(u)+\mathrm{content}(v)$ for all words
$u,v$.

Throughout, $\sigma_1=\begin{pmatrix}0&1\\1&0\end{pmatrix}$ and
$\sigma_3=\begin{pmatrix}1&0\\0&-1\end{pmatrix}$ denote the two real
symmetric Pauli matrices; they satisfy $\sigma_1^2=\sigma_3^2=I$ and
$\sigma_1\sigma_3=-\sigma_3\sigma_1$.

\begin{lemma}\label{lem:content-cyclic}
For each content $q$, the sum of the coefficients of the words of content $q$
is a cyclic invariant: if $f\cyc g$ then $f$ and $g$ have the same content
sums. Moreover the coefficient of the monomial $x^ay^b$ in $\ab f$ is the
content sum of $f$ at $q=(a,b)$.
\end{lemma}

\begin{proof}
A commutator $uv-vu$ has zero coefficient sum within each content, since $uv$
and $vu$ have the same content and opposite coefficients. The second statement
is the definition of the abelianization.
\end{proof}

Lemma~\ref{lem:content-cyclic} is what allows the abelianization to be read off
a Gram matrix even though $f$ is only cyclically equivalent to a sum of
squares; it is used without further comment in Section~\ref{sec:kred}.

\subsection{The space $Z$}

\begin{lemma}\label{lem:dimZ}
$\dim Z = 6$, with two dimensions in each of the contents $(4,2)$, $(3,3)$ and
$(2,4)$, and none in $(6,0)$, $(5,1)$, $(1,5)$, $(0,6)$.
\end{lemma}

\begin{proof}
By \S\ref{sec:words} the symmetric part of the degree-six homogeneous
polynomials modulo $\cyc$ has the bracelet classes as coordinates. These
number thirteen, distributed by content as $1,1,3,3,3,1,1$ for
$(6,0),\dots,(0,6)$. (Under rotation alone the content $(3,3)$ carries four
classes; the classes of $X^2YXY^2$ and $X^2Y^2XY$ are exchanged by reversal,
so they merge on the symmetric part, and the fourteenth cyclic dimension is
the antisymmetric line of Remark~\ref{rem:needsym}.) Vanishing of the
abelianization imposes one linear condition per content by
Lemma~\ref{lem:content-cyclic}, and each condition is non-trivial, leaving
$13-7=6$.
\end{proof}

\begin{remark}\label{rem:needsym}
The hypothesis $f^*=f$ in \eqref{eq:Z} cannot be dropped. Put
\[
  h \;=\; X^2YXY^2 - X^2Y^2XY .
\]
Then $h$ is homogeneous of degree six and $\ab h\equiv0$, both words
collapsing to $x^3y^3$. Since $\tr w(A,B)=\tr w^*(A,B)$ for symmetric $A,B$
and the two words are reversals of one another up to rotation,
$\tr h(A,B)=0$ identically; in particular $h$ is trace-positive. Yet
$h\notin\Th$. Indeed cyclic equivalence commutes with the involution, because
$(uv-vu)^*=-(u^*v^*-v^*u^*)$, and $\sum_i g_i^*g_i$ is symmetric; so
$h\cyc\sum_i g_i^*g_i$ would give $h^*\cyc h$. But reversal interchanges the
two cyclic classes occurring in $h$, so $h^*\cyc-h$, whence $2h\cyc0$ --
impossible, since $X^2YXY^2$ and $X^2Y^2XY$ are distinct cyclic classes.
Thus in the quotient of \emph{all} degree-six polynomials with vanishing
abelianization, which has dimension $14-7=7$, trace-positivity does not imply
membership in $\Th$: the extra direction is invisible to the trace and lies
outside $\Th$ for reasons of symmetry alone. Restricting to
$\SymA\R\langle X,Y\rangle$ removes exactly that direction, and is also the
standard setting in which trace-positivity is compared with Hermitian-square
certificates.
\end{remark}

\section{Reduction to a four-dimensional Gram matrix}\label{sec:kred}

\begin{definition}  \label{def:K}
Let
\[
  k_1 = X^2Y - XYX,\quad k_2 = XYX - YX^2,\quad
  k_3 = XY^2 - YXY,\quad k_4 = YXY - Y^2X ,
\]
and let $K=\operatorname{span}\{k_1,k_2,k_3,k_4\}$.
\end{definition}

Each $k_i$ has vanishing abelianization, and $k_1,k_2$ have content $(2,1)$
while $k_3,k_4$ have content $(1,2)$. In fact $K$ is intrinsic: the homogeneous
cubics form an eight-dimensional space, abelianization maps it onto the
four-dimensional span of $x^3,x^2y,xy^2,y^3$, so its kernel has dimension four,
and the four linearly independent elements $k_i$ lie in that kernel. Hence
\begin{equation}\label{eq:Kintrinsic}
  K=\{g\in\R\langle X,Y\rangle:\ g\ \text{homogeneous of degree }3,\
  \ab g\equiv0\},
\end{equation}
so that $K$ is characterized intrinsically, not by the particular basis
$k_1,\dots,k_4$; this is what allows clause~(iv) of Theorem~\ref{thm:main} to
be stated as a condition on the $g_i$ themselves.

\begin{lemma}[K-reduction]\label{lem:kred}
Let $f\in Z\cap\Th$. Then $f\cyc\sum_i g_i^*g_i$ with every $g_i\in K$.
\end{lemma}

\begin{proof}
\emph{Step 1: the degrees are bounded.} Suppose $f\cyc\sum_i g_i^*g_i$, put
$D=\max_i\deg g_i$, and let $G_i$ denote the degree-$D$ homogeneous component
of $g_i$. If $D>3$ then $2D>6=\deg f$. A commutator of homogeneous
elements is homogeneous of degree $\deg u+\deg v$, so the span of the
commutators is a graded subspace and each homogeneous component of a sum of
commutators is again one; comparing degree-$2D$ components of $f$ and
$\sum_i g_i^*g_i$ therefore gives $\sum_i G_i^*G_i\cyc0$. Taking traces at
any $A,B\in\Sym_n(\R)$ annihilates the commutators and leaves
$\sum_i\|G_i(A,B)\|_F^2=0$, so every $G_i$ vanishes on all symmetric pairs of
every size and hence $G_i=0$, contradicting the choice of $D$. Hence
$\deg g_i\le3$ for all $i$.

\emph{Step 2: the $g_i$ may be taken homogeneous of degree three.} Write
$f\cyc v^*Gv$ with $G\succeq0$ indexed by the fifteen words of degree at most
three.
Comparing degree-six components as in Step 1 gives $f\cyc v_3^*G_{33}v_3$,
where $v_3$ lists the eight words of degree three and $G_{33}$ is the
corresponding principal submatrix of $G$, itself positive semidefinite.
Since $G_{33}\succeq0$ and $f\cyc v_3^*G_{33}v_3$, we may work with $G_{33}$ in
place of $G$: factor $G_{33}=\sum_i c_ic_i^\top$ (for instance by Cholesky) and
set $g_i=c_i^\top v_3$; these are the $g_i$ the lemma produces, and
$f\cyc\sum_i g_i^*g_i$.

\emph{Step 3: the coefficients.} Consider the content $(6,0)$. Since
$\mathrm{content}(u^*v)=\mathrm{content}(u)+\mathrm{content}(v)$, a product
$u^*v$ with $u,v$ of degree three has content $(6,0)$ only if $u=v=X^3$, so by
Lemma~\ref{lem:content-cyclic} the content sum of $f$ at $(6,0)$ equals
$G[X^3,X^3]=\sum_i c_i[X^3]^2$. This vanishes because $\ab f\equiv0$, so
$c_i[X^3]=0$ for every $i$
and the whole row and column of $X^3$ in $G$ vanish. Symmetrically for $Y^3$.

Now consider $(4,2)$. Products $u^*v$ of that content arise from the three
splittings $(\deg_Xu,\deg_Xv)\in\{(1,3),(2,2),(3,1)\}$; the outer two force
$v=X^3$, respectively $u=X^3$, and so involve the row of $X^3$, already zero. Hence
only $u,v$ of content $(2,1)$ contribute, i.e.\ the principal $3\times3$ block
$G_W$ on $W=\{X^2Y, XYX, YX^2\}$, and the content sum of $f$ at $(4,2)$ is
$\mathbf1^\top G_W\mathbf1=\|G_W^{1/2}\mathbf1\|^2$. It vanishes, so
$G_W\mathbf1=0$; and from $G=\sum_i c_ic_i^\top$ we get
$0=\mathbf1^\top G_W\mathbf1=\sum_i\langle c_i|_W,\mathbf1\rangle^2$, so every
$g_i$ has coefficients summing to zero on $W$. The span of the vectors in
$\R^W$ orthogonal to $\mathbf1$ is exactly $\operatorname{span}\{k_1,k_2\}$.
Symmetrically the content-$(1,2)$ part of each $g_i$ lies in
$\operatorname{span}\{k_3,k_4\}$. With the $(3,0)$ and $(0,3)$ parts already
zero, this gives $g_i\in K$.
\end{proof}

\begin{corollary}\label{cor:gram}
$\Th\cap Z=\Th_3\cap Z$, and it is the image of the cone of positive
semidefinite $4\times4$ matrices under
$H\mapsto \sum_{i,j}H_{ij}\,k_i^*k_j$. Writing
\begin{equation}\label{eq:Hblocks}
  H=\begin{pmatrix} P & R\\ R^\top & Q\end{pmatrix},\qquad P,Q,R\in\R^{2\times2},
\end{equation}
the block $P$ produces the content-$(4,2)$ part of $f$, the block $Q$ the
content-$(2,4)$ part, and the off-diagonal block $R$ the content-$(3,3)$ part.
Moreover $H\succeq0$ if and only if $P,Q\succeq0$ and $R=P^{1/2}CQ^{1/2}$ for
some $C$ with $\|C\|\le1$; this is the standard factorization of positive
semidefinite block matrices due to \v{S}mul'jan~\cite{Smuljan} (see
also Albert's generalized Schur complement criterion~\cite{Albert}), and it
requires no invertibility assumption on $P$ or $Q$.
\end{corollary}

The same map, taken over all symmetric $H$ rather than only the positive
semidefinite ones, is \emph{onto} $Z$ -- note that $\sum_{i,j}H_{ij}k_i^*k_j$
is symmetric whenever $H$ is, so its image lies in $Z$ to begin with. This
follows from
Lemmas~\ref{lem:42alg} and~\ref{lem:33alg} below, which exhibit the image of
each block as the full corresponding sector. Hence every $f\in Z$ has
coordinates $(s,u)$, $(S,U)$, $(\alpha,\beta)$ in the sense used throughout
-- with $m=(u+2s)/4$ and $n=(U+2S)/4$ the derived Gram parameters of
\S\ref{sec:42} -- whether or not it lies in $\Th$.

\section{Decoupling the content sectors}\label{sec:sectors}

Write $f=f_{42}+f_{33}+f_{24}$ for the decomposition of $f\in Z$ by content,
and $T_q(A,B)=\tr f_q(A,B)$.

\begin{lemma}\label{lem:decouple}
Fix $N$. Then $\tr f(A,B)\ge0$ for all $A,B\in\Sym_N(\R)$ if and only if, for
all $A,B\in\Sym_N(\R)$,
\[
  T_{42}\ge0,\qquad T_{24}\ge0,\qquad |T_{33}|\le 2\sqrt{T_{42}T_{24}} .
\]
\end{lemma}

\begin{proof}
Substituting $(\lambda A,\mu B)$ multiplies $T_{42}$, $T_{33}$, $T_{24}$ by
$\lambda^4\mu^2$, $\lambda^3\mu^3$, $\lambda^2\mu^4$ respectively. Dividing
$\lambda^4\mu^2T_{42}+\lambda^3\mu^3T_{33}+\lambda^2\mu^4T_{24}$ by
$\lambda^2\mu^4>0$ and putting $t=\lambda/\mu$, non-negativity for all
$\lambda,\mu\ne0$ becomes
\[
  T_{42}\,t^2+T_{33}\,t+T_{24}\ \ge\ 0 \qquad\text{for all }t\ne0,
\]
hence for all $t\in\R$ by continuity. Setting $t=0$ gives $T_{24}\ge0$, and
dividing by $t^2$ and letting $t\to\infty$ gives $T_{42}\ge0$. A quadratic
with non-negative leading and constant coefficients is non-negative on $\R$
if and only if $T_{33}^2\le4T_{42}T_{24}$, which is the stated condition. The
converse is immediate.
\end{proof}

Corollary~\ref{cor:gram} exhibits the same shape on the $\Th$ side: $P,Q\succeq0$
with $R=P^{1/2}CQ^{1/2}$ for a contraction $C$. The proof of Theorem~\ref{thm:main} amounts to matching the two.

\section{The sector \texorpdfstring{$(4,2)$}{(4,2)}}\label{sec:42}

Let $c_1=\langle X^4Y^2\rangle$, $c_2=\langle X^3YXY\rangle$,
$c_3=\langle X^2YX^2Y\rangle$ be the three bracelet classes of content $(4,2)$.

\begin{lemma}\label{lem:42alg}
Recall $k_1=X^2Y-XYX$ and $k_2=XYX-YX^2$ (Definition~\ref{def:K}). Then
$k_1^*k_1 = k_2^*k_2 = c_1-2c_2+c_3$ and $k_1^*k_2=k_2^*k_1=2c_2-2c_3$.
Consequently, with $s=\tr P$ and $m=P_{12}$,
\[
  f_{42} \;=\; s\,c_1 + (4m-2s)\,c_2 + (s-4m)\,c_3 .
\]
\end{lemma}

\begin{proof}
Expanding, $k_1^*k_1=(YX^2-XYX)(X^2Y-XYX)$ equals
\[
  YX^4Y-YX^3YX-XYX^3Y+XYX^2YX ,
\]
whose four words lie in the bracelet classes $c_1,c_2,c_2,c_3$ respectively,
giving $c_1-2c_2+c_3$; the computation for $k_2^*k_2$ is symmetric. Similarly
$k_1^*k_2=(YX^2-XYX)(XYX-YX^2)$ expands to
\[
  YX^3YX-YX^2YX^2-XYX^2YX+XYXYX^2 .
\]
The first and last words rotate to $X^3YXY$ and the middle two to $X^2YX^2Y$,
so $k_1^*k_2=2c_2-2c_3$; the same computation gives $k_2^*k_1$.

Finally, since $H$ is symmetric and $k_1^*k_1=k_2^*k_2$,
\[
  f_{42}=P_{11}k_1^*k_1+2P_{12}k_1^*k_2+P_{22}k_2^*k_2
        =(P_{11}+P_{22})(c_1-2c_2+c_3)+2P_{12}(2c_2-2c_3),
\]
and substituting $P_{11}+P_{22}=s$, $P_{12}=m$ gives
$f_{42}=s\,c_1+(4m-2s)\,c_2+(s-4m)\,c_3$.
\end{proof}

Since the coefficients of $f_{42}$ sum to zero (Lemma~\ref{lem:content-cyclic}), the sector is
two-dimensional, parametrized by $(s,u)$ with $u$ the coefficient of $c_2$ and
the coefficient of $c_3$ equal to $-s-u$; equivalently $m=(u+2s)/4$.

\begin{proposition}\label{prop:42}
For $f_{42}$ in the content-$(4,2)$ part of $Z$, with coordinates $(s,u)$, the
following are equivalent:
\begin{enumerate}
  \item[(a)] $f_{42}$ arises from some $P\succeq0$;
  \item[(b)] $T_{42}\ge0$ on all symmetric pairs;
  \item[(c)] $T_{42}\ge0$ on all pairs in $\Sym_2(\R)$;
  \item[(d)] $(s,u)$ lies in
  \begin{equation}\label{eq:42cone}
    \{\,s\ge0,\qquad -4s\le u\le 0\,\}.
  \end{equation}
\end{enumerate}
\end{proposition}

\begin{proof}
A positive semidefinite $P$ with $\tr P=s$ and $P_{12}=m$ exists if and only if
$s\ge0$ and $|m|\le s/2$. Necessity: $m^2\le P_{11}P_{22}$ because the
$2\times2$ determinant of a positive semidefinite matrix is non-negative, and
$P_{11}P_{22}\le(s/2)^2$ by the arithmetic--geometric mean inequality.
Sufficiency: take $P=\tfrac s2 I+m\sigma_1$, whose eigenvalues are
$\tfrac s2\pm m$. (For a \emph{given} $P$ these two conditions are not
sufficient; it is the existence of some such $P$ that is at issue.) By
Lemma~\ref{lem:42alg} this region is exactly \eqref{eq:42cone}, and
$\Th\subseteq\{\text{trace-positive}\}$ gives one inclusion. For the converse it suffices to exhibit, for each of the two facets,
a pair in $\Sym_2(\R)$ whose trace functional is the corresponding outer
normal.

\emph{Facet $u\le0$.} Take the Pauli pair $A=a\sigma_3$, $B=b\sigma_1$. Then
$A^2=a^2I$, $B^2=b^2I$ and $(\sigma_3\sigma_1)^2=-I$, so the traces of $c_1,c_2,c_3$ are
$2a^4b^2$, $-2a^4b^2$, $2a^4b^2$. Hence
\[
  T_{42}=s\,(2a^4b^2)+u\,(-2a^4b^2)+(-s-u)(2a^4b^2)=-4a^4b^2\,u ,
\]
so $T_{42}\ge0$ on this pair forces $u\le0$.

\emph{Facet $4s+u\ge0$.} Take $A=I+\varepsilon M$ with $M,B\in\Sym_2(\R)$. All
three traces agree to order $\varepsilon^0$ and $\varepsilon^1$, being
$\tr(B^2)$ and $4\tr(MB^2)$ respectively for every $M$ and $B$; since the
three coefficients $s,u,-s-u$ sum to zero, those two orders cancel
identically in $T_{42}$. At order $\varepsilon^2$, writing
$\pi_{\!M}=\tr(M^2B^2)$ and $\kappa_{\!M}=\tr(MBMB)$,
they are $6\pi_{\!M}$, $3\kappa_{\!M}+3\pi_{\!M}$ and $4\kappa_{\!M}+2\pi_{\!M}$,
so that
\[
  T_{42}(I+\varepsilon M,B)
  \;=\;\varepsilon^2\Big[s\,6\pi_{\!M}+u\,(3\kappa_{\!M}+3\pi_{\!M})
        +(-s-u)(4\kappa_{\!M}+2\pi_{\!M})\Big]+O(\varepsilon^3)
\]
\[
  \phantom{T_{42}(I+\varepsilon M,B)}
  \;=\;\varepsilon^2(\pi_{\!M}-\kappa_{\!M})(4s+u)+O(\varepsilon^3).
\]
Dividing by $\varepsilon^2>0$ and letting $\varepsilon\to0$, non-negativity of
$T_{42}$ forces $(\pi_{\!M}-\kappa_{\!M})(4s+u)\ge0$. Now
\[
  \pi_{\!M}-\kappa_{\!M}=\tfrac12\tr\big([M,B]^\top[M,B]\big)\ \ge 0,
\]
with equality only if $M$ and $B$ commute. Taking $M=\sigma_3$, $B=\sigma_1$
gives $\pi_{\!M}-\kappa_{\!M}=4>0$, and therefore $4s+u\ge0$.

The remaining inequality $s\ge0$ follows from the other two: $u\le0$ and
$4s+u\ge0$ give $4s\ge-u\ge0$.
\end{proof}

The sector $(2,4)$ is obtained from $(4,2)$ by exchanging $X$ and $Y$; we write
$S=\tr Q$, $n=Q_{12}$ and $U$ for the coefficient of $\langle XYXY^3\rangle$,
so that $n=(U+2S)/4$.

\section{The sector \texorpdfstring{$(3,3)$}{(3,3)}}\label{sec:33}

\subsection{The cross-block algebra}

Let $d_1=\langle X^3Y^3\rangle$, $d_2=\langle X^2YXY^2\rangle$,
$d_3=\langle XYXYXY\rangle$ be the three bracelet classes of content $(3,3)$.

\begin{lemma}\label{lem:33alg}
Recall $k_1=X^2Y-XYX$, $k_2=XYX-YX^2$, $k_3=XY^2-YXY$, $k_4=YXY-Y^2X$
(Definition~\ref{def:K}). Then
$k_1^*k_3=k_2^*k_4=d_1-2d_2+d_3$ and $k_1^*k_4=k_2^*k_3=d_2-d_3$.
Hence the map $R\mapsto f_{33}$ has rank two and factors through
\[
  \alpha=\tfrac12\tr R,\qquad \beta=\tfrac12\tr(R\sigma_1).
\]
Explicitly, since $f_{33}=2\sum_{i\le2<j}H_{ij}k_i^*k_j$,
\begin{equation}\label{eq:f33}
  f_{33} \;=\; 4\alpha\,(d_1-2d_2+d_3)\;+\;4\beta\,(d_2-d_3).
\end{equation}
\end{lemma}

The two polynomials on the right of \eqref{eq:f33} are combinations of the
content-$(3,3)$ classes $d_1,d_2,d_3$ whose coefficients sum to zero, so both
lie in the content-$(3,3)$ part of $Z$ (Lemma~\ref{lem:content-cyclic}). They
are linearly independent, and that part is two-dimensional by
Lemma~\ref{lem:dimZ}, so they span it. Consequently every $f\in Z$ determines
a unique pair $(\alpha,\beta)$ through \eqref{eq:f33}, and
$R=\alpha I+\beta\sigma_1$ is one block producing it. Below, $(\alpha,\beta)$
always denotes these coordinates.

\begin{remark}
Under cyclic equivalence alone the sector carries four classes and the
identities read $k_1^*k_3=e_1-e_2-e_3+e_4$, $k_1^*k_4=2e_2-e_3-e_4$ and
$k_2^*k_3=-e_2+2e_3-e_4$, where $e_1,\dots,e_4$ are the cyclic classes of
$X^3Y^3$, $X^2YXY^2$, $X^2Y^2XY$, $XYXYXY$ respectively. The displayed form is
their image under the reversal-induced identification $e_2=e_3$
(\S\ref{sec:words}), so that $d_1=e_1$, $d_2=e_2=e_3$, $d_3=e_4$. The
identification is legitimate precisely because $Z$ consists of symmetric
elements; the antisymmetric combination $e_2-e_3$ is the direction excluded in
Remark~\ref{rem:needsym}.
\end{remark}

\subsection{Removing the residual Gram freedom}\label{sec:union}

By Lemma~\ref{lem:42alg}, $f_{42}$ depends on $P$ only through $s=\tr P$ and
$m=P_{12}$. Hence $f_{42}$ determines $P$ only up to the remaining parameter
$\delta=\tfrac12(P_{11}-P_{22})$, which may take any value compatible with
$P\succeq0$; similarly, $f_{24}$ determines $Q$ only up to
$\delta'=\tfrac12(Q_{11}-Q_{22})$. A larger $R$ may therefore be admissible for
one choice of $(\delta,\delta')$ than for another, and the correct statement of
$\Th\cap Z$ involves the union over all admissible choices. We show the union
is attained at $\delta=\delta'=0$; as noted in Remark~\ref{rem:notneeded}, the
proof of Theorem~\ref{thm:main} does not depend on this.

The argument uses two maximizations, which should be kept apart. For fixed
$(\delta,\delta')$ and a fixed direction $M$, maximizing over the contractions
$C$ gives the support function $h_{\delta,\delta'}(M)$ of the attainable set,
i.e.\ how far that set extends in the direction $M$. We then maximize $h_{\delta,\delta'}(M)$ over the Gram freedom
$(\delta,\delta')$ and show that $\delta=\delta'=0$ attains the maximum in
every direction (Proposition~\ref{prop:ridge}), so that no other choice
enlarges the attainable set (Corollary~\ref{cor:union}).

Write
\[
  P=\tfrac s2 I+m\sigma_1+\delta\sigma_3,\qquad
  Q=\tfrac S2 I+n\sigma_1+\delta'\sigma_3 ,
\]
and let $M=\lambda I+\mu\sigma_1$ represent, through $R\mapsto\tr(MR)$, a
linear functional on the two-dimensional space of attainable $f_{33}$. The
support function of the admissible set of $R$ in the direction $M$ is
\[
  h_{\delta,\delta'}(M)
  \;=\;\max_{\|C\|\le1}\ \langle M, P^{1/2}CQ^{1/2}\rangle
  \;=\;\big\|Q^{1/2}MP^{1/2}\big\|_* ,
\]
the nuclear norm.

\begin{lemma}\label{lem:reparam}
Let $P,Q\succeq0$ and $M=\lambda I+\mu\sigma_1$. Then
$\big\|Q^{1/2}MP^{1/2}\big\|_*=\big\|(MQM)^{1/2}P^{1/2}\big\|_*$. Moreover,
writing $Q'=MQM=\tfrac{S'}{2}I+n'\sigma_1+\delta''\sigma_3$, we have
\[
  \tfrac{S'}{2}=\tfrac S2(\lambda^2+\mu^2)+2n\lambda\mu,\qquad
  n'=S\lambda\mu+n(\lambda^2+\mu^2),\qquad
  \delta''=(\det M)\,\delta' ,
\]
and $\det Q'=(\det M)^2\det Q$.
\end{lemma}

\begin{proof}
For any matrix $A$ the squared singular values of $A$ are the eigenvalues of
$A^\top A$. Since $M$ is symmetric, both $A=Q^{1/2}MP^{1/2}$ and
$A=(MQM)^{1/2}P^{1/2}$ give $A^\top A=P^{1/2}MQMP^{1/2}$; hence the two
matrices have the same singular values and the same nuclear norm.

For $Q'$, write $Q=\tfrac S2I+n\sigma_1+\delta'\sigma_3$ and use
$\sigma_1^2=\sigma_3^2=I$, $\sigma_1\sigma_3=-\sigma_3\sigma_1$. Since $\sigma_1$
commutes with $M$,
\[
  M^2=(\lambda^2+\mu^2)I+2\lambda\mu\,\sigma_1,\qquad
  M\sigma_1M=\sigma_1M^2=2\lambda\mu\,I+(\lambda^2+\mu^2)\sigma_1 ,
\]
and since $\sigma_3M=(\lambda I-\mu\sigma_1)\sigma_3$,
\[
  M\sigma_3M=M(\lambda I-\mu\sigma_1)\sigma_3=(\lambda^2-\mu^2)\sigma_3
  =(\det M)\,\sigma_3 .
\]
Collecting the coefficients of $I$, $\sigma_1$ and $\sigma_3$ in
$Q'=\tfrac S2M^2+nM\sigma_1M+\delta'M\sigma_3M$ gives the three formulas.
Finally $\det Q'=(\det M)^2\det Q$ by multiplicativity of the determinant.
\end{proof}

The essential point is that $S'$ and $n'$ do not involve $\delta'$ at all: the
reparametrization rescales $\delta'$ and leaves the rest of $Q$ alone.

\begin{proposition}\label{prop:ridge}
Fix a direction $M=\lambda I+\mu\sigma_1$, and let
$h_{\delta,\delta'}(M)=\big\|Q^{1/2}MP^{1/2}\big\|_*$ be the support function
defined above. Put $\rho=\tfrac{s^2}{4}-m^2$ and
$\rho'=(\det M)^2\big(\tfrac{S^2}{4}-n^2\big)$. If $\rho'>0$ then
\begin{equation}\label{eq:cosine}
  h_{\delta,\delta'}(M)^2 \;=\; c_0 + 2\sqrt{\rho\rho'}\,\cos(\varphi-\psi),
  \qquad \delta=\sqrt \rho\,\sin\varphi,\quad \delta''=\sqrt{\rho'}\sin\psi ,
\end{equation}
with $c_0$ depending on $M$ but not on $\delta$ or $\delta''$, so that
$h_{\delta,\delta'}(M)$ is maximized exactly on the ridge
$\delta/\sqrt \rho=\delta''/\sqrt{\rho'}$; if $\rho'=0$ then
$h_{\delta,\delta'}(M)$ does not depend on $\delta$ at all. In either case
$\delta=\delta'=0$ is optimal.
\end{proposition}

\begin{proof}
By Lemma~\ref{lem:reparam} we may work with $Q'=MQM$ in place of $Q$, so that
the quantity to be maximized is $\|Q'^{1/2}P^{1/2}\|_*$.
For positive semidefinite $2\times2$ matrices,
$\|Q'^{1/2}P^{1/2}\|_*=\sqrt{\tr(PQ')+2\sqrt{\det P\det Q'}}$, from
$\tr\sqrt{N}=\sqrt{\tr N+2\sqrt{\det N}}$ applied to $N=P^{1/2}Q'P^{1/2}$.
Since $I,\sigma_1,\sigma_3$ are mutually orthogonal for the pairing
$(A,B)\mapsto\tr(AB)$, each with $\tr(\,\cdot\,^2)=2$, only matching
components contribute to
\[
  \tr(PQ')=\tfrac{sS'}{2}+2mn'+2\delta\delta'' ,
\]
so the $\delta$-dependence is carried entirely by the cross term. Also
$\det P=\rho-\delta^2$ and $\det Q'=\rho'-\delta''^2$, and admissibility of
$(\delta,\delta')$ means exactly $|\delta|\le\sqrt\rho$ and
$|\delta''|\le\sqrt{\rho'}$.

If $\rho'=0$ then $\delta''=0$ and $\det Q'=0$, so
$h_{\delta,\delta'}(M)^2=\tfrac{sS'}{2}+2mn'$ is independent of $\delta$ and
there is nothing to prove; this covers in particular the degenerate direction
$\det M=0$. If $\rho=0$ then $\delta=0$ is the only admissible value, and
again there is nothing to prove. So assume $\rho,\rho'>0$. Substituting the
sine parametrization, with $\varphi,\psi\in[-\pi/2,\pi/2]$ so that the cosines
are non-negative, turns
$2\delta\delta''+2\sqrt{(\rho-\delta^2)(\rho'-\delta''^2)}$ into
$2\sqrt{\rho\rho'}(\sin\varphi\sin\psi+\cos\varphi\cos\psi)$, which is
\eqref{eq:cosine}. This is maximal exactly when $\varphi=\psi$. Finally
$\rho'>0$ forces $\det M\ne0$, so $\delta''=0$ if and only if $\delta'=0$, and
the origin lies on the ridge.
\end{proof}

\begin{corollary}\label{cor:union}
For admissible $(\delta,\delta')$ -- those for which $P,Q\succeq0$ -- write
$\mathcal S_{\delta,\delta'}\subset\R^2$ for the set of $(\alpha,\beta)$
attainable from $R=P^{1/2}CQ^{1/2}$ with $\|C\|\le1$. Then
\[
  \bigcup_{\delta,\delta'}\mathcal S_{\delta,\delta'}\;=\;\mathcal S_{0,0}.
\]
\end{corollary}

\begin{proof}
Each $\mathcal S_{\delta,\delta'}$ is the image of the operator-norm unit ball
under a linear map, hence compact, convex and symmetric about the origin;
its support function is $h_{\delta,\delta'}$, as defined before
Lemma~\ref{lem:reparam}. Proposition~\ref{prop:ridge}
says that $\sup_{\delta,\delta'}h_{\delta,\delta'}=h_{0,0}$ pointwise in the
direction. Since the support function of the closed convex hull of a union is
the supremum of the support functions,
$\operatorname{conv}\bigcup_{\delta,\delta'}\mathcal S_{\delta,\delta'}$ and
$\mathcal S_{0,0}$ have the same support function, and being compact convex
sets they are equal. Hence
\[
  \mathcal S_{0,0}\subseteq\bigcup_{\delta,\delta'}\mathcal S_{\delta,\delta'}
  \subseteq\operatorname{conv}\bigcup_{\delta,\delta'}\mathcal S_{\delta,\delta'}
  =\mathcal S_{0,0},
\]
and all three coincide. The support-function identity alone would only show
that the convex hull of the union equals $\mathcal S_{0,0}$; since the union
need not be convex a priori, it is the inclusion
$\mathcal S_{0,0}\subseteq\bigcup_{\delta,\delta'}\mathcal S_{\delta,\delta'}$
that proves equality.
\end{proof}

\begin{remark}\label{rem:notneeded}
Corollary~\ref{cor:union} is not needed for the proof of
Theorem~\ref{thm:main}. That proof constructs a single positive semidefinite
$H$ and takes $\delta=\delta'=0$ from the outset, using the $P$ and $Q$ built
in the proof of Proposition~\ref{prop:42}; the rectangle
\eqref{eq:rect} attached to $f_{42},f_{24}$ is then determined by $f$ alone,
because $\delta=0$ fixes $P$ once $s$ and $m$ are known. What
Corollary~\ref{cor:union} adds is a direct proof, on the $\Th$ side alone,
that no other admissible $(\delta,\delta')$ enlarges the attainable set; that
\eqref{eq:rect} describes $\Th\cap Z$ exactly also follows from
Theorem~\ref{thm:main}, since $f\in\Th$ is trace-positive by
\eqref{eq:inclusion} and Proposition~\ref{prop:33} then places $(\alpha,\beta)$
in the rectangle. Section~\ref{sec:union} may therefore be skipped: no other result of the paper depends on it.
\end{remark}

\subsection{The attainable rectangle}\label{sec:rect}

From now on we take $\delta=\delta'=0$; by Corollary~\ref{cor:union} this
loses nothing, and the proof of Theorem~\ref{thm:main} does not need even
that (Remark~\ref{rem:notneeded}). Thus
\begin{equation}\label{eq:PQcommute}
  P=\tfrac s2 I+m\sigma_1,\qquad Q=\tfrac S2 I+n\sigma_1 .
\end{equation}
Both then lie in the commutative algebra generated by $I$ and $\sigma_1$, hence
are simultaneously diagonal in the eigenbasis of $\sigma_1$, with eigenvalues
\begin{equation}\label{eq:pq}
  p_\pm=\tfrac s2\pm m\ \ge 0,\qquad q_\pm=\tfrac S2\pm n\ \ge 0 .
\end{equation}

\begin{lemma}\label{lem:rectangle}
With $P,Q$ as in \eqref{eq:PQcommute}, the set of $(\alpha,\beta)$ attainable
from $R=P^{1/2}CQ^{1/2}$ with $\|C\|\le1$ is the rectangle
\begin{equation}\label{eq:rect}
  |\alpha+\beta|\le\sqrt{p_+q_+},\qquad |\alpha-\beta|\le\sqrt{p_-q_-} .
\end{equation}
\end{lemma}

\begin{proof}
The set is convex and symmetric about the origin, so it is determined by its
support function. For a direction $M=\lambda I+\mu\sigma_1$ the pairing is
$\tr(MR)=2\lambda\alpha+2\mu\beta$, and
\[
  \max_{\|C\|\le1}\ \tr\!\big(M P^{1/2}CQ^{1/2}\big)
  =\big\|Q^{1/2}MP^{1/2}\big\|_* .
\]
In the $\sigma_1$ eigenbasis $Q^{1/2}MP^{1/2}$ is diagonal with entries
$\sqrt{p_\pm q_\pm}\,(\lambda\pm\mu)$, so its nuclear norm is
$\sqrt{p_+q_+}\,|\lambda+\mu|+\sqrt{p_-q_-}\,|\lambda-\mu|$. Writing
$u=\lambda+\mu$ and $v=\lambda-\mu$, the support function in the coordinates
$(\alpha+\beta,\alpha-\beta)$ is $\sqrt{p_+q_+}\,|u|+\sqrt{p_-q_-}\,|v|$; as
$u$ and $v$ range independently over $\R$, this is the support function of the
rectangle \eqref{eq:rect}.
\end{proof}

\begin{proposition}\label{prop:33}
Let $f=f_{42}+f_{33}+f_{24}$, with $f_{42}$ and $f_{24}$ produced by $P$ and
$Q$ as in \eqref{eq:PQcommute}. If $(\alpha,\beta)$ lies outside the rectangle
\eqref{eq:rect} then $f$ is not trace-positive, and a witness may be taken in
$\Sym_2(\R)$.
\end{proposition}

\begin{proof}
Two matrix families suffice. In both, $A$ and $B$ lie in
$\operatorname{span}\{I,\sigma_1,\sigma_3\}$, where $\sigma_1^2=\sigma_3^2=I$
and $\sigma_1\sigma_3=-\sigma_3\sigma_1$. Expanding a word of length six in
$A,B$ therefore gives a linear combination of $I$, $\sigma_1$, $\sigma_3$ and
$\sigma_1\sigma_3$; since $\tr\sigma_1=\tr\sigma_3=\tr(\sigma_1\sigma_3)=0$, its
trace is twice the coefficient of $I$. Computing these coefficients gives the
identities below.

\emph{Family A.} For $A=-(aI+b\sigma_3)$ and $B=cI+d\sigma_1$,
\begin{equation}\label{eq:famA}
  \tr f(A,B)\;=\;16\,b^2d^2\Big[\,a^2p_++c^2q_+-2ac(\alpha+\beta)
  \;+\;b^2p_-+d^2q_-\,\Big].
\end{equation}
Put $b=d=\varepsilon>0$. The quadratic form $a^2p_++c^2q_+-2ac(\alpha+\beta)$
in $(a,c)$ has matrix
$\left(\begin{smallmatrix}p_+&-(\alpha+\beta)\\
-(\alpha+\beta)&q_+\end{smallmatrix}\right)$,
which fails to be positive semidefinite precisely when
$(\alpha+\beta)^2>p_+q_+$. In that case choose $(a,c)\in\R^2$ making it
negative -- the sign of $a$ is free, so both facets
$|\alpha+\beta|=\sqrt{p_+q_+}$ are reached by this one family -- and then
$\varepsilon$ small enough that the whole bracket is negative; since
$b^2d^2=\varepsilon^4>0$ we get $\tr f(A,B)<0$.

\emph{Family B.} For $A=t\sigma_3$ and $B=w\sigma_3+r\sigma_1$,
\begin{equation}\label{eq:famB}
  \tr f(A,B)\;=\;16\,t^2r^2\Big[\,t^2p_-+w^2q_--2tw(\alpha-\beta)
  \;+\;r^2q_-\,\Big],
\end{equation}
and the same argument applies with $r$ in the role of $\varepsilon$: the form
in $(t,w)$ fails to be positive semidefinite exactly when
$(\alpha-\beta)^2>p_-q_-$, and $t,w$ again range over all of $\R$, so both
facets are reached. Note that whenever that form is negative we have $t\ne0$,
since $w^2q_-\ge0$; hence the prefactor $16t^2r^2$ is strictly positive.

The two families exhaust the complement of \eqref{eq:rect}.
\end{proof}
\begin{remark}\label{rem:geometric}
The optimal scales are $a=\sqrt{q_+}$, $c=\sqrt{p_+}$ in family A and
$t=\sqrt{q_-}$, $w=\sqrt{p_-}$ in family B, at which the brackets reduce to
$2\sqrt{p_+q_+}\big(\sqrt{p_+q_+}-(\alpha+\beta)\big)+O(\varepsilon^2)$ and
$2\sqrt{p_-q_-}\big(\sqrt{p_-q_-}-(\alpha-\beta)\big)+O(r^2)$. The thresholds
are \emph{geometric} means of the eigenvalues. The symmetric choice $a=c$,
$t=w$ would give the arithmetic means $\tfrac12(p_++q_+)$ and
$\tfrac12(p_-+q_-)$ instead, and is sharp only when $p_+=q_+$, respectively
$p_-=q_-$ -- in particular when $P=Q=I$. This is why each family must scale $A$ and $B$ independently.
\end{remark}

\begin{remark}
Both families are required. The symmetries available on $Z$ are
$(\alpha,\beta)\mapsto(-\alpha,-\beta)$, induced by $Y\mapsto-Y$, and the
exchange $X\leftrightarrow Y$, which sends $k_1\mapsto-k_4$, $k_2\mapsto-k_3$
and therefore fixes each of $d_1,d_2,d_3$: it acts trivially on
$(\alpha,\beta)$ while interchanging $p_\pm$ with $q_\pm$. Neither maps the
pair of facets $|\alpha+\beta|=\sqrt{p_+q_+}$ to
$|\alpha-\beta|=\sqrt{p_-q_-}$. Structurally, $R=\alpha I+\beta\sigma_1$ has eigenvalues
$\alpha\pm\beta$, and family A activates only the direction
$k_1+k_2\leftrightarrow k_3+k_4$ while family B activates only
$k_1-k_2\leftrightarrow k_3-k_4$.
\end{remark}

\begin{remark}\label{rem:zeromin}
Inside the rectangle nothing needs proving: $f\in\Th$ there, so $f$ is
trace-positive by~\eqref{eq:inclusion}. That the minimum is exactly zero
is also immediate, since every $f\in Z$ has zero trace on
commuting pairs.
\end{remark}

\section{Proof of Theorem~\ref{thm:main}}\label{sec:proof}

(iii)$\Rightarrow$(iv) is Lemma~\ref{lem:kred} together with
Corollary~\ref{cor:gram}: a positive semidefinite $4\times4$ matrix factors
into at most four rank-one terms, giving at most four squares. (iv)$\Rightarrow$(iii) is trivial, and
(iii)$\Rightarrow$(i)$\Rightarrow$(ii) is \eqref{eq:inclusion}.

It remains to prove (ii)$\Rightarrow$(iii). Let $f\in Z$ satisfy
$\tr f(A,B)\ge0$ on $\Sym_2(\R)$. By Lemma~\ref{lem:decouple} with $N=2$ we
obtain $T_{42}\ge0$ and $T_{24}\ge0$ on $\Sym_2(\R)$. By
Proposition~\ref{prop:42} this places $f_{42}$ and $f_{24}$ in the cones
\eqref{eq:42cone}. Take for $P$ and $Q$ the matrices constructed in the proof
of that proposition, namely $P=\tfrac s2I+m\sigma_1$ and
$Q=\tfrac S2I+n\sigma_1$ as in \eqref{eq:PQcommute}, with eigenvalues
$p_\pm,q_\pm\ge0$ as in \eqref{eq:pq}. (This is the case $\delta=\delta'=0$;
Corollary~\ref{cor:union} is not needed here, see Remark~\ref{rem:notneeded}.)

Proposition~\ref{prop:33} now forces $(\alpha,\beta)$ into the rectangle
\eqref{eq:rect}: otherwise one of the families A, B produces a pair in
$\Sym_2(\R)$ with negative trace. By Lemma~\ref{lem:rectangle} there is
therefore $C$ with $\|C\|\le1$ such that $R=P^{1/2}CQ^{1/2}$ produces
$f_{33}$, and by Corollary~\ref{cor:gram} the resulting $H$ is positive
semidefinite. Hence $f\in\Th_3$. \qed

\begin{remark}
The proof of (ii)$\Rightarrow$(iii) never leaves $\Sym_2(\R)$: every
certificate used -- the Pauli pair, the perturbation $I+\varepsilon M$, and
the two families \eqref{eq:famA}, \eqref{eq:famB} -- is a pair of $2\times2$
matrices, and the sector-decoupling scaling $X\mapsto tX$ preserves matrix
size. This is the source of clause (ii), and matches the corresponding clause
of the degree-four theorem.
\end{remark}

Written out in the coordinates of $f$, the cone is defined by a handful of inequalities.

\begin{corollary}\label{cor:explicit}
Let $f\in Z$ have coordinates $(s,u)$, $(S,U)$ and $(\alpha,\beta)$ in the
three content sectors. Then $f$ is trace-positive -- equivalently
$f\in\Th_3$ -- if and only if
\begin{equation}\label{eq:explicit1}
  4s+u\ \ge\ 0,\qquad u\ \le\ 0,\qquad 4S+U\ \ge\ 0,\qquad U\ \le\ 0,
\end{equation}
\begin{equation}\label{eq:explicit2}
  16(\alpha+\beta)^2\ \le\ (4s+u)(4S+U),
  \qquad
  16(\alpha-\beta)^2\ \le\ uU .
\end{equation}
\end{corollary}

\begin{proof}
By Lemma~\ref{lem:42alg} we have $m=(u+2s)/4$, so \eqref{eq:pq} gives
\[
  p_+=\tfrac{4s+u}{4},\qquad p_-=-\tfrac u4 ,
\]
and symmetrically $q_+=(4S+U)/4$, $q_-=-U/4$. Thus \eqref{eq:explicit1} says
exactly that $p_\pm\ge0$ and $q_\pm\ge0$, which by
Proposition~\ref{prop:42} is the sector condition \eqref{eq:42cone}; the
inequality $s\ge0$ is immediate because $s=p_++p_-$. Given
\eqref{eq:explicit1}, both right-hand sides of \eqref{eq:explicit2} are
non-negative, so squaring is reversible and \eqref{eq:explicit2} is exactly
the rectangle \eqref{eq:rect}. Theorem~\ref{thm:main} identifies the
conjunction with trace-positivity.
\end{proof}

\begin{remark}
The eigenvalues of $P$ are therefore the two facet functions of the sector
cone, and those of $Q$ likewise. This is the cleanest form of the matching
between the two sides of the problem: the extreme sectors fix the spectra of
the diagonal blocks, and the middle sector is confined by their geometric
means.
\end{remark}

\begin{remark}\label{rem:verified}
The algebraic identities on which the proof rests -- Lemmas~\ref{lem:42alg}
and~\ref{lem:33alg}, the \mbox{$\varepsilon$-expansion} in
Proposition~\ref{prop:42}, and the two families \eqref{eq:famA},
\eqref{eq:famB} -- are finite computations in the free algebra and can be
checked by hand; each has also been re-verified symbolically over exact
arithmetic. We add one practical caution for anyone computing in this cone
numerically rather than symbolically. Just outside the boundary of
\eqref{eq:rect} the quantity $\min\tr f$ can be extremely small: at a relative
overshoot of $10^{-3}$ in the rectangle thresholds we have seen it at
$10^{-12}$, and a multistart minimization in double precision then returns
machine zero and reports $f$ as trace-positive, unchanged by raising the
number of restarts from $250$ to $4000$. The semidefinite program on the
Gram matrix of Corollary~\ref{cor:gram} has no such difficulty, returning a
smallest eigenvalue of order $-10^{-4}$ on the same instances. Near the
boundary the closed form of Corollary~\ref{cor:explicit} or the semidefinite
program should be preferred to a search.
\end{remark}

\section{Open questions}\label{sec:open}

\begin{enumerate}
\item The value of $d_{\min}$ -- defined for \emph{symmetric} elements, see
      \S\ref{sec:settle} -- remains open in $4<d_{\min}\le12$.
      Theorem~\ref{thm:main} removes the zero-abelianization route at degree
      six; the remaining possibilities at that degree have non-zero
      abelianization which is nevertheless a sum of squares.
\item Does the analogue of Theorem~\ref{thm:main} hold in more variables?
      The argument of the K-reduction lemma adapts, though with more contents
      to handle, while the sector structure and the two-dimensionality of
      the $(3,3)$ image are special to two letters. In three variables the gap
      already appears in degree four~\cite{Quarez}, so the analogous question
      there concerns quartics rather than sextics.
\item Homogenization is unavailable in the noncommutative setting (see
      \cite{Bur-thesis}, and the remarks in \cite{Quarez}), so the
      degree-four theorem for two letters does not settle homogeneous
      $f\in\R\langle X_1,X_2,X_3\rangle_4$, the analogue of Hilbert's
      ternary quartics; the examples of \cite{Quarez} are inhomogeneous. To
      our knowledge it is not known whether every trace-positive such $f$
      lies in $\Th$. The methods here are sector-based and may bear on that
      question.
\end{enumerate}

\section*{Declarations}

\noindent\textbf{Funding.} No funding was received for conducting this study.

\medskip
\noindent\textbf{Competing interests.} The author has no competing interests
to declare.

\medskip
\noindent\textbf{Data and code availability.} No datasets were generated or
analyzed. All results are established analytically; the identities of
Sections~\ref{sec:42} and~\ref{sec:33} are finite computations that can be
verified by hand. The symbolic verification scripts and the numerical
experiment reported in Remark~\ref{rem:verified} are available from the author
on request.

\end{document}